\documentclass[12pt,reqno]{amsproc}%
\usepackage{mathrsfs, amsfonts, amsthm, amsmath, amssymb}
\usepackage{amsmath}
\usepackage{amsfonts}
\usepackage{amssymb}
\usepackage{graphicx}%
\providecommand{\U}[1]{\protect\rule{.1in}{.1in}}
\theoremstyle{plain}
\newtheorem{theorem}{Theorem}[subsection]

\newtheorem{corollary}[theorem]{Corollary}

\newtheorem{definition}[theorem]{Definition}

\newtheorem{lemma}[theorem]{Lemma}

\newtheorem{proposition}[theorem]{Proposition}
\newtheorem{remark}[theorem]{Remark}

\makeatletter
\newcommand{\LeftEqNo}{\let\veqno\@@leqno}

\makeatother
\numberwithin{equation}  {section}
\begin{document}
\title[Some Results on Inner Quasidiagonal C*-algebras]{Some Results on Inner Quasidiagonal C*-algebras}
\author{Qihui Li}
\curraddr{School of Science, East China University of Science and Technology, Shanghai,
200237, P. R. China}
\email{qihui\_li@126.com}
\thanks{The authors were partially supported by NSFC(Grant No.11671133).}

\begin{abstract}
In the current article, we prove the cross product C*-algebra by a Rokhlin
action of finite group on a strongly quasidiagonal C*-algbra is strongly
quasidiagonal again. We also show that a just-infinite C*-algebra is
quasidiagonal if and only if it is inner quasidiagonal. Finally, we compute
the topological free entropy dimension in just-infinite C*-algebras.

\end{abstract}
\maketitle

\ 

\vspace{-0.1cm}




\section{Introduction}

\smallskip For distinguish the class of NF algebras and the class of strong NF
algebras, Blackadar and Kirchberg introduced the concept of inner
quasidiagonal C*-algebras in \cite{BK2}. From its definition, it is apparent
that the class of inner quasidiagonal C*-algebras is a subclass of
quasidiagonal C*-algebras. Many basic properties of inner quasidiagonal
C*-algebras have been discussed in \cite{BK2} and \cite{BK3}. It was also
shown that a separable C*-algebra is a strong NF algebra if and only if it is
nuclear and inner quasidiagonal. Therefore the class of all strong NF algebras
is strictly contained in the class of nuclear and quasidiagonal C*-algebras
(i.e. NF algebras). Examples of separable nuclear C*-algebras which are
quasidiagonal but not inner quasidiagonal were given in the same article. And
we also know that all separable simple C*-algebras are inner quasidiagonal,
all strongly quasidiagonal C*-algebras are inner quasidiagonal. Recall that a
C*-algebra is called strongly quasidiagonal if it is separable and all its
representations are quasidiagonal.

Since not all C*-subalgebras of inner quasidiagonal C*-algebra are inner, the
crossed product C*-algebra by an action of a finite group on an inner
quasidiagonal C*-algebra may not be inner quasidiagonal again. In \cite{WZ},
it was shown that the crossed product C*-algebra by a Rokhlin action of a
finite group on a unital inner quasidiagonal C*-algebra is inner again. So it
is natural to ask whether the same conclusion still hold for non-unital inner
quasidiagonal C*-algebras. In the current paper, we will prove that the
crossed product C*-algebra by a Rokhlin action of finite group on a strongly
quasidiagonal C*-algebra (may not be unital) is strongly quasidiagonal again.

Just-infinite C*-algebras are first introduced by R. Grigorchuk, M. Musat and
M. R$\phi$dam in \cite{GMR} as an analogous notion of the well-established
notions of just-infinite groups and just-infinite (abstract) algebras. A
C*-algebra is called just-infinite if it is infinite-dimensional and all its
proper quotient are finite-dimensional. The necessary and sufficient
conditions for a C*-algebra to be just-infinite was also given in the same
article when the algebras are separable. By using the characterizations of the
just-infinite C*-algebras given in \cite{GMR}, we will prove that every
just-infinite C*-algebra is quasidiagonal if and only if it is inner
quasidiagonal in the separable case. We also show that the crossed product
C*-algebra by a Rokhlin action of a finite group on a just-infinite C*-algebra
(may not be unital) is just-infinite again.

At last, we analyze the topological free entropy dimension in just-infinite
C*-algebras. The notion of topological free entropy dimension $\delta_{top}$
of $n$-tuples of elements in a unital C* algebra was introduced by Voiculescu
in \cite{DV}, where some basic properties of free entropy dimension are
discussed. For more information about $\delta_{top},$ we refer the readers to
\cite{HS2}, \cite{HLS} and \cite{HLSW}.

The organization of the paper is as follows. In section 2, we first recall
some definitions and fix some notation. Then we prove that the crossed product
C*-algebra by a Rokhlin action of finite group on a strongly quasidiagonal
C*-algebra is strongly quasidiagonal again. At last, we prove that every
just-infinite C*-algebra is quasidiagonal if and only if it is inner
quasidiagonal. In section 3, we compute the topological free entropy dimension
$\delta_{top}$ in just-infinite C*-algebras.

\section{Inner Quasidiagonal C*-algebras}

\subsection{\ Definitions and Preliminaries}

The concept of inner quasidiagonal C*-algebras was given by Blackadar and
Kirchberg in \cite{V}, which is obtained by slightly modifying the
Voiculescu's characterization of quasidiagonal C*-algebras.

\begin{definition}
(\cite{BK2}) \textit{A C*-algebra }$\mathcal{A}$\textit{\ is inner
quasidiagonal if, for every }$x_{1},\cdots,x_{n}$ in $\mathcal{A}%
$\textit{\ and }$\varepsilon>0,$\textit{\ there is a representation }$\pi
$\textit{\ of }$\mathcal{A}$\textit{\ on a Hilbert space }$\mathcal{H}%
$\textit{, and a finite-rank projection }$P\in\pi\left(  \mathcal{A}\right)
^{\prime\prime}$\textit{\ such that }$\left\Vert P\pi\left(  x_{i}\right)
-\pi\left(  x_{i}\right)  P\right\Vert <\varepsilon,\left\Vert P\pi\left(
x_{i}\right)  P\right\Vert >\left\Vert x_{i}\right\Vert -\varepsilon
$\textit{\ for }$1\leq i\leq n.$
\end{definition}

Next result give us a characterization of inner quasidiagonal C*-algebras.

\begin{theorem}
\label{11b}(\cite{BK3}) A separable C*-algebra is inner quasidiagonal if and
only if it has a separating family of quasidiagonal irreducible representations.
\end{theorem}

Now let us recall the concept of completely positive maps. A map $\varphi$
from C*-algebra $\mathcal{A}$ to a C*-algebra $\mathcal{B}$ is said to be
completely positive if $\varphi_{n}:\mathcal{M}_{n}\left(  \mathcal{A}\right)
\rightarrow\mathcal{M}_{n}(\mathcal{B}),$ defined by
\[
\varphi_{n}\left(  \left[  a_{i,j}\right]  \right)  =\left[  \varphi\left(
a_{i,j}\right)  \right]  ,
\]
is positive for every $n.$ We use c.p. to abbreviate "completely positive,"
u.c.p. for "unital completely positive" and c.c.p. for "contractive completely positive."

\begin{definition}
Let $\mathcal{A}$ and $\mathcal{B}$ be C*-algebras and $\varphi
:\mathcal{A\rightarrow B}$ be a c.c.p. map. The C*-subalgebra
\[
\mathcal{A}_{\varphi}=\left\{  a\in\mathcal{A}:\varphi\left(  a^{\ast
}a\right)  =\varphi\left(  a\right)  ^{\ast}\varphi(a)\text{ and }%
\varphi\left(  aa^{\ast}\right)  =\varphi\left(  a\right)  \varphi\left(
a\right)  ^{\ast}\right\}
\]
is called the multiplicative domain of $\varphi.$ It is well-known that
$\mathcal{A}_{\varphi}$ is the largest subalgebra of $\mathcal{A}$ on which
$\varphi$ restricts to a $\ast$-homomorphism.
\end{definition}

Suppose $\pi:\mathcal{B}\left(  \mathcal{H}\right)  \rightarrow\mathcal{B}%
\left(  \mathcal{H}\right)  /\mathbb{K}\mathcal{(H)}$ is the canonical mapping
onto the Calkin algebra.

\begin{lemma}
\label{dixmier}(Stinespring) Let $\mathcal{A}$ be a unital C*-algebra and
$\varphi:\mathcal{A\longrightarrow B}\left(  \mathcal{H}\right)  $ be a c.p.
map. Then, there exist a Hilbert space $\widehat{\mathcal{H}}$, a $\ast
$-representation $\pi_{\varphi}:\mathcal{A\longrightarrow B}\left(
\widehat{\mathcal{H}}\right)  $ and an operator $V:\mathcal{H\longrightarrow
}\widehat{\mathcal{H}}$ such that
\[
\varphi\left(  a\right)  =V^{\ast}\pi_{\varphi}\left(  a\right)  V
\]
for every $a\in\mathcal{A}$. In particular, $\left\Vert \varphi\right\Vert
=\left\Vert V^{\ast}V\right\Vert =\left\Vert \varphi\left(  1\right)
\right\Vert .$
\end{lemma}

We call the triplet $\left(  \pi_{\varphi},\widehat{\mathcal{H}},V\right)  $
in preceding lemma a Stinespring dilation of $\varphi.$ When $\varphi$ is
unital, $V^{\ast}V=\varphi\left(  I\right)  =I,$ and hence $V$ is an isometry.
So in this case we may assume that $\mathcal{H}\subseteq\widehat{\mathcal{H}}$
and $V$ is a projection $P$ on $\mathcal{H}$ such that $\varphi\left(
a\right)  =P\pi_{\varphi}\left(  a\right)  |_{\mathcal{H}}$.

So next result give us another characterization of inner quasidiagonal C*-algebras.

\begin{theorem}
\label{inner 2}(\cite{BO})A unital C*-algebra $\mathcal{A}$ is inner
quasidiagonal if and only if there is a sequence of u.c.p. maps $\varphi
_{n}:\mathcal{A}\rightarrow\mathcal{M}_{k_{n}}(\mathbb{C})$ such that
$\left\Vert a\right\Vert =\lim\left\Vert \varphi_{n}\left(  a\right)
\right\Vert $ and $dist(a,\mathcal{A}_{\varphi_{n}})\rightarrow0$ for all
$a\in\mathcal{A}$ where $\mathcal{A}_{\varphi_{n}}$ is the multiplicative
domain of $\varphi_{n}.$
\end{theorem}

\smallskip Recall that a faithful representation of a C*-algebra $\mathcal{A}$
is called essential if $\pi\left(  \mathcal{A}\right)  $ contains no nonzero
finite rank operators

\begin{proposition}
\label{Voi} (1.7, \cite{V2}) Let $\pi:\mathcal{A}\rightarrow\mathcal{B}\left(
\mathcal{H}\right)  $ be a faithful essential representation. Then
$\mathcal{A}$ is quasidiagonal if and only if $\pi\left(  \mathcal{A}\right)
$ is quasidiagonal set of operators.
\end{proposition}

A C*-algebra $\mathcal{A}$ is called antiliminal if $\mathcal{A}$ contains no
nonzero abelian elements, i.e. if every nonzero hereditary C*-subalgebra of
$\mathcal{A}$ is noncommutative. Recall that a C*-algebra is called primitive
if it admit a faithful irreducible representation. It is said to be prime if,
whenever $\mathcal{I}$ and $\mathcal{J}$ are closed two-sided ideals in
$\mathcal{A}$ such that $\mathcal{I\cap J=}$ $0,$ then either $\mathcal{I=}$
$0$, or $\mathcal{J=}$ $0.$ It is well-known that a separable C*-algebra is
prime if and only if it is primitive. Combining Theorem \ref{11b}, Proposition
\ref{Voi} and IV.1.1.7 in \cite{BK1}, we can quickly get the following lemma.

\begin{lemma}
\label{inner 1.2}(Corollary 2.6 in \cite{BK2}) Every separable antiliminal
quasidiagonal prime C*-algebra is inner. Every separable simple quasidiagonal
C*-algebra is inner quasidiagonal.
\end{lemma}

\begin{remark}
\label{inner 1}In \cite{BK2}, it has been show that an arbitrary inductive
limit (with injective connecting maps) of inner quasidiagonal C*-algebras is
inner quasidiagonal. Every residually finite-dimensional C*-algebra is inner quasidiagonal.
\end{remark}

\subsection{Crossed product C*-algebras by actions of finite groups}

It is well-known that the subalgebras of inner quasidiagonal may not be inner
quasidiagonal again. It implies that the crossed product C*-algebras by an
action of a finite group on an inner quasidiagonal C*-algebra may not be inner
quasidiagonal again.

\begin{remark}
\label{2}Let $\mathcal{A\subseteq B}\left(  \mathcal{H}\right)  $ be a
separable inner quasidiagonal C*-algebra and $\alpha:G\rightarrow Aut\left(
\mathcal{A}\right)  $ be an action of finite group $G$ on $\mathcal{A}$.
Suppose $\pi:\mathcal{A\rightarrow B}\left(  \mathcal{K}\right)  $ is a $\ast
$-homomorphism and $\left\{  e_{g,q}\right\}  $ is the family of canonical
matrix units of $\mathcal{B}\left(  l^{2}\left(  G\right)  \right)  $. Then we
define an action $\beta:G\rightarrow Aut(\pi\left(  \mathcal{A}\right)  )$ by
letting
\[
\beta_{g}\left(  \pi\left(  a\right)  \right)  =\pi\left(  \alpha
_{g}(a)\right)  .
\]
Therefore it is not hard to verify that the mapping $\rho:\mathcal{A\rtimes
}_{\alpha}G\rightarrow\mathcal{\pi}\left(  \mathcal{A}\right)
\mathcal{\rtimes}_{\beta}G$ by letting
\[
\rho\left(  I_{\mathcal{H}}\otimes\lambda_{g}\right)  =I_{\mathcal{K}}%
\otimes\lambda_{g}\text{ and }\rho\left(  \sum_{g\in G}\alpha_{g}^{-1}\left(
a\right)  \otimes e_{g,g}\right)  =\sum_{g\in G}\beta_{g}^{-1}\left(
\pi\left(  a\right)  \right)  \otimes e_{g,g}.
\]
is a $\ast$-homomorphism again.
\end{remark}

\begin{theorem}
\label{cha2}Let $\mathcal{A}$ be a unital separable C*-algebra. If there are a
sequence of u.c.p maps $\varphi_{n}:\mathcal{A}\rightarrow\mathcal{M}_{k_{n}%
}(\mathbb{C})$ such that $\left\Vert a\right\Vert =\lim\left\Vert \varphi
_{n}\left(  a\right)  \right\Vert $ and $d(a,\mathcal{A}_{\varphi_{n}%
})\rightarrow0$ for all $a\in\mathcal{A}$ and $\alpha:G\rightarrow\alpha_{g}$
an automorphic representation of finite group $G$ on $\mathcal{A}$ satisfying
that $\mathcal{A}_{\varphi_{n}}$ is $\alpha_{g}$ invariant (i.e. $\alpha
_{g}|_{\mathcal{A}_{\varphi_{n}}}$ is an automorphic of $\mathcal{A}%
_{\varphi_{n}})$ for each $n$ and $g\in G,$ then $\mathcal{A\rtimes}_{\alpha
}G$ is inner quasidiagonal again.
\end{theorem}

\begin{proof}
By Lemma \ref{dixmier}, for each $\varphi_{n},$ there are a Hilbert space
$\mathcal{H}_{n}$ containing a copy of $\mathbb{C}^{k_{n}}$ and a
representation $\pi_{n}:\mathcal{A\rightarrow B}\left(  \mathcal{H}%
_{n}\right)  $ such that $\varphi_{n}\left(  a\right)  =P_{n}\pi_{n}\left(
a\right)  P_{n}$ for $a\in\mathcal{A}$ where $P_{n}$ is the finite-rank
projection from $\mathcal{H}_{n}$ onto $\mathbb{C}^{k_{n}}.$ Then there is a
$\ast$-homomorphism
\[
\rho_{n}:\mathcal{A\rtimes}_{\alpha}G\rightarrow\pi_{n}\left(  \mathcal{A}%
\right)  \mathcal{\rtimes}_{\beta_{n}}G
\]
where $\beta_{n}$ and $\rho_{n}$ are defined in the way introduced by Remark
\ref{2}. Note $P_{n}\otimes I_{|G|}$ is a finite-rank projection where
$I_{|G|}$ is the identity in $\mathcal{M}_{\left\vert G\right\vert }\left(
\mathbb{C}\right)  $ and then the map $\widehat{\varphi}_{n}$ defined by the
form
\[
\widehat{\varphi}_{n}\left(  c\right)  =\left(  P_{n}\otimes I_{|G|}\right)
\rho_{n}\left(  c\right)  \left(  P_{n}\otimes I_{|G|}\right)  \text{ for
every }c\in\mathcal{A\rtimes}_{\alpha}G
\]
is a u.c.p. map. It is obvious that $I_{\mathcal{A}}\otimes\lambda_{g}%
\in\left(  \mathcal{A\rtimes}_{\alpha}G\right)  _{\widehat{\varphi}_{n}}$
where $\left(  \mathcal{A\rtimes}_{\alpha}G\right)  _{\widehat{\varphi}_{n}}$
is the multiplicative domain of $\widehat{\varphi}_{n}$. For $a,b\in
\mathcal{A}_{\varphi_{n}}$, we have $\alpha_{g}(a)$ and $\alpha_{g}\left(
b\right)  $ are both in $\mathcal{A}_{\varphi_{n}}$ by the fact that
$\alpha_{g}|_{\mathcal{A}_{\varphi_{n}}}$ is an automorphism of $\mathcal{A}%
_{\varphi_{n}}$ for each $n$ and $g\in G,$ then
\begin{align*}
&  \widehat{\varphi}_{n}\left(  \sum_{g\in G}\alpha_{g}^{-1}\left(  a\right)
\otimes e_{g,g}\cdot\sum_{g\in G}\alpha_{g}^{-1}\left(  b\right)  \otimes
e_{g,g}\right)  =\left(  P_{n}\otimes I_{|G|}\right)  \rho_{n}\left(
\sum_{g\in G}\alpha_{g}^{-1}\left(  ab\right)  \otimes e_{g,g}\right)  \left(
P_{n}\otimes I_{|G|}\right) \\
&  =\sum_{g\in G}P_{n}\beta_{g}^{-1}\left(  \pi\left(  ab\right)  \right)
P_{n}\otimes e_{g,g}=\sum_{g\in G}P_{n}\pi\left(  \alpha_{g}^{-1}\left(
ab\right)  \right)  P_{n}\otimes e_{g,g}\\
&  =(\sum_{g\in G}P_{n}(\beta_{g}^{-1}(\pi(a))P_{n}\otimes e_{g,g})(\sum_{g\in
G}P_{n}(\beta_{g}^{-1}(\pi(a))P_{n}\otimes e_{g,g})\\
&  =\widehat{\varphi}_{n}(\sum_{g\in G}\alpha_{g}^{-1}\left(  a\right)
\otimes e_{g,g})\widehat{\varphi}_{n}(\sum_{g\in G}\alpha_{g}^{-1}\left(
b\right)  \otimes e_{g,g})
\end{align*}
and it is easily to check
\[
\widehat{\varphi}_{n}\left(  \left(  \sum_{g\in G}\alpha_{g}^{-1}\left(
a\right)  \otimes e_{g,g}\right)  \cdot\left(  I_{\mathcal{A}}\otimes
\lambda_{g}\right)  \right)  =\widehat{\varphi}_{n}\left(  \sum_{g\in G}%
\alpha_{g}^{-1}\left(  a\right)  \otimes e_{g,g}\right)  \widehat{\varphi}%
_{n}\left(  I_{\mathcal{A}}\otimes\lambda_{g}\right)  .
\]
Therefore%
\begin{equation}
\left(  \mathcal{A\rtimes}_{\alpha}G\right)  _{\widehat{\varphi}_{n}}%
\supseteq\mathcal{A}_{\varphi_{n}}\mathcal{\rtimes}_{\alpha}G\text{ for every
}g\in G. \tag{2.1}%
\end{equation}
Since $\left(  I\otimes\lambda_{g}\right)  \left(  I\otimes\lambda_{h}\right)
=I\otimes\lambda_{gh}$ and
\[
\left(  I\otimes\lambda_{h}\right)  \left(  \sum_{g\in G}\alpha_{g}%
^{-1}\left(  a\right)  \otimes e_{g,g}\right)  \left(  I\otimes\lambda
_{h}^{\ast}\right)  =\sum_{g\in G}\alpha_{g}^{-1}\left(  \alpha_{h}\left(
a\right)  \right)  \otimes e_{g,g},
\]
we only need to prove $dist\left(  c,\left(  \mathcal{A\rtimes}_{\alpha
}G\right)  _{\widehat{\varphi}_{n}}\right)  \rightarrow0$ and $\lim
\sup\left\Vert \widehat{\varphi}_{n}\left(  c\right)  \right\Vert =\left\Vert
c\right\Vert $ as $c$ in the form
\[
c=\left(  \sum_{g\in G}\alpha_{g}^{-1}\left(  a\right)  \otimes e_{g,g}%
\right)  \left(  I\otimes\lambda_{h}\right)
\]
without the loss of generality. Then by Theorem \ref{inner 2},
$\mathcal{A\rtimes}_{\alpha}G$ is inner quasidiagonal.

Now for $c=(\sum_{g\in G}\alpha_{g}^{-1}\left(  a\right)  \otimes
e_{g,g})\left(  I\otimes\lambda_{h}\right)  ,$ we have $dist(a,\mathcal{A}%
_{\varphi_{n}})\rightarrow0$ for all $a\in\mathcal{A}$ from the assumption$.$
Then there are $a^{\left(  n\right)  }\in\mathcal{A}_{\varphi_{n}}$ for
$n\in\mathbb{N}$ such that $\left\Vert a-a^{\left(  n\right)  }\right\Vert
\rightarrow0.$ Note
\[
\left(  \sum_{g\in G}\alpha_{g}^{-1}\left(  a^{\left(  n\right)  }\right)
\otimes e_{g,g}\right)  \left(  I\otimes\lambda_{h}\right)  \in\left(
\mathcal{A\rtimes}_{\alpha}G\right)  _{\widehat{\varphi}_{n}}%
\]
by (2.1). Then
\begin{align*}
&  \left\Vert c-\left(  \sum_{g\in G}\alpha_{g}^{-1}\left(  a^{\left(
n\right)  }\right)  \otimes e_{g,g}\right)  \left(  I\otimes\lambda
_{h}\right)  \right\Vert \\
&  =\left\Vert \left(  \sum_{g\in G}\alpha_{g}^{-1}\left(  a-a^{(n)}\right)
\otimes e_{g,g}\right)  \left(  I\otimes\lambda_{h}\right)  \right\Vert
\leq\left\Vert a-a^{(n)}\right\Vert \rightarrow0.
\end{align*}
Therefore $dist\left(  c,\left(  \mathcal{A\rtimes}_{\alpha}G\right)
_{\widehat{\varphi}_{n}}\right)  \rightarrow0$. By $\left\Vert a\right\Vert
=\lim\sup\left\Vert \varphi_{n}\left(  a\right)  \right\Vert $ in the
assumption$,$ we also have%
\begin{align*}
\lim\sup\left\Vert \widehat{\varphi}_{n}\left(  c\right)  \right\Vert  &
=\lim\sup\left\Vert \left(  P_{n}\otimes I_{|G|}\right)  \left(  \sum_{g\in
G}\pi\left(  \alpha_{g}^{-1}\left(  a\right)  \right)  \otimes e_{g,g}\right)
\left(  P_{n}\otimes I_{|G|}\right)  \widehat{\varphi}_{n}\left(
I\otimes\lambda_{h}\right)  \right\Vert \\
&  =\lim\sup\left\Vert \sum_{g\in G}\varphi_{n}\left(  \alpha_{g}^{-1}\left(
a\right)  \right)  \otimes e_{g,g}\right\Vert =\left\Vert \sum_{g\in G}\left(
\alpha_{g}^{-1}\left(  a\right)  \right)  \otimes e_{g,g}\right\Vert \\
&  =\left\Vert \sum_{g\in G}\left(  \alpha_{g}^{-1}\left(  a\right)  \right)
\otimes e_{g,g}\left(  I\otimes\lambda_{h}\right)  \right\Vert =\left\Vert
c\right\Vert
\end{align*}
Hence by Lemma \ref{inner 2}, $\mathcal{A\rtimes}_{\alpha}G$ is inner
quasidiagonal again.
\end{proof}

Even though the result in Theorem \ref{cha2} give us the condition to ensure
the crossed product of inner quasidiagonal C*-algebras to be inner
quasidiagonal again, but the condition is too strong. In \cite{WZ}, Rokhlin
actions have been considered for unital inner quasidiagonal C*-algebras. In
this section, we are going to consider the Rokhlin actions on strongly
quasidiagonal C*-algebras where C*-algebras may not be unital. Note every
strongly quasidiagonal C*-algebra is inner quasidiagonal.

\begin{lemma}
\label{strong 1}Let $\mathcal{A}$ be a strongly quasidiagonal C*-algebra and
$\mathcal{B}\subseteq\mathcal{A}$ a hereditary C*-subalgebra. Then
$\mathcal{B}$ is strongly quasidiagonal again.
\end{lemma}

\begin{proof}
By Proposition 2.8 in \cite{BK2}, we only need to show $\pi\left(
\mathcal{B}\right)  $ is a quasidiagonal set of operators for any irreducible
representation $\pi$ of $\mathcal{B}$.

Let $\pi:\mathcal{B}\rightarrow\mathcal{B}\left(  \mathcal{H}\right)  $ be an
irreducible representation of $\mathcal{B}.$ Then we can extend $\pi$ to an
irreducible representation $\widetilde{\pi}:\mathcal{A}\rightarrow
\mathcal{B}\left(  \widetilde{\mathcal{H}}\right)  $ with $\mathcal{H}%
\subseteq\widetilde{\mathcal{H}}$ and $\widetilde{\pi}\left(  x\right)
|_{\mathcal{H}}=\pi\left(  x\right)  $ for all $x\in\mathcal{B}$ by
Proposition 2.10.2 in \cite{Dix}. From the fact that $\mathcal{A}$ is strongly
quasidiagonal, we know $\widetilde{\pi}\left(  \mathcal{A}\right)  $ is a
quasidiagonal set of operators and then $\widetilde{\pi}\left(  \mathcal{B}%
\right)  $ is a quasidiagonal set of operators too. Since $\mathcal{B}$ is
hereditary, by II.6.1.9 in \cite{BK1} $\widetilde{\pi}\left(  \mathcal{B}%
\right)  |_{\widehat{\mathcal{H}}}$ is irreducible where $\widehat
{\mathcal{H}}\subseteq$ $\widetilde{\mathcal{H}}$ is the essential subspace
(see II.6.1.5 in \cite{BK1}) of $\widetilde{\pi}\left(  \mathcal{B}\right)  .$
Therefore
\[
\widetilde{\pi}\left(  \mathcal{B}\right)  =\widetilde{\pi}\left(
\mathcal{B}\right)  |_{\widehat{\mathcal{H}}}\oplus0|_{\widehat{\mathcal{H}%
}^{\bot}}.
\]
Note the projection $p_{\mathcal{H}}$ on $\mathcal{H}$ is in $\widetilde{\pi
}\left(  \mathcal{B}\right)  ^{\prime}$, we also have
\[
\widetilde{\pi}\left(  \mathcal{B}\right)  =\pi\left(  \mathcal{B}\right)
\oplus\widetilde{\pi}\left(  \mathcal{B}\right)  |_{\mathcal{H}^{\bot}%
}\text{.}%
\]
Since $\mathcal{H\subseteq}$ $\widehat{\mathcal{H}}$ by the fact that is
$\widehat{\mathcal{H}}$ the essential subspace of $\widetilde{\pi}\left(
\mathcal{B}\right)  $ and $\widetilde{\pi}\left(  \mathcal{B}\right)
|_{\widehat{\mathcal{H}}}$ is irreducible, we have $\widehat{\mathcal{H}%
}=\mathcal{H}$ and $\widetilde{\pi}\left(  \mathcal{B}\right)  |_{\mathcal{H}%
^{\bot}}=0|_{\mathcal{H}^{\bot}}$. It implies that
\[
\widetilde{\pi}\left(  \mathcal{B}\right)  =\pi\left(  \mathcal{B}\right)
\oplus0|_{\mathcal{H}^{\bot}}.
\]

So if $\mathcal{H}^{\bot}=\{0\}$, then $\widetilde{\pi}\left(  \mathcal{B}%
\right)  =\pi\left(  \mathcal{B}\right)  $ is a quasidiagonal set of
operators. Otherwise, by Lemma 15.5 in \cite{BN}, we know that $\pi\left(
\mathcal{B}\right)  $ is a quasidiagonal set of operators on $\mathcal{H}$.
Hence $\mathcal{B}$ is strongly quasidiagonal C*-algebra.
\end{proof}

It is natural to ask whether the similar conclusion holds for inner
quasidiagonal C*-algebras.

\begin{proposition}
\label{inner here}Let $\mathcal{A}$ be a separable C*-algebra. If
$\mathcal{A}$ is inner quasidiagonal and $\mathcal{B}\subseteq\mathcal{A}$ is
a hereditary C*-subalgebra, then $\mathcal{B}$ is inner quasidiagonal again.
\end{proposition}

\begin{proof}
By Theorem \ref{11b}, we only need to prove there is a separating family of
irreducible quasidiagonal representations of $\mathcal{B}.$

Let $\left\{  \widetilde{\pi}_{n}\right\}  $ be be a separating family of
irreducible quasidiagonal representations of $\mathcal{A}$ on $\widetilde
{\mathcal{H}}_{n}.$ Then $\widetilde{\pi}_{n}\left(  \mathcal{B}\right)
\subseteq\widetilde{\pi}_{n}\left(  \mathcal{A}\right)  $ is a quasidiagonal
set of operators in $\mathcal{B}\left(  \widetilde{\mathcal{H}}_{n}\right)  $.
By II.6.1.9 in \cite{BK1}, $\widetilde{\pi}_{n}\left(  \mathcal{B}\right)
|_{\mathcal{H}_{n}}$ is irreducible where $\mathcal{H}_{n}\subseteq$
$\widetilde{\mathcal{H}}_{n}$ is the essential subspace of $\widetilde{\pi
}_{n}\left(  \mathcal{B}\right)  .$ Therefore $\widetilde{\pi}_{n}\left(
\mathcal{B}\right)  =\widetilde{\pi}_{n}\left(  \mathcal{B}\right)
|_{\mathcal{H}_{n}}\oplus0|_{\mathcal{H}_{n}^{\bot}}.$ So by Lemma 3.10 in
\cite{BN}, we have $\widetilde{\pi}_{n}\left(  \mathcal{B}\right)
|_{\mathcal{H}_{n}}$ is a quasidiagonal set of operators in $\mathcal{B}%
\left(  \mathcal{H}_{n}\right)  $. It implies that $\widetilde{\pi}%
_{n}|_{\mathcal{H}_{n}}$ is a separating family of irreducible quasidiagonal
representation of $\mathcal{B}.$ Therefore $\mathcal{B}$ is inner
quasidiagonal again.
\end{proof}

A closed two-sided ideal $\mathcal{I}$ in a C*-algebra $\mathcal{A}$ is said
to be primitive if $\mathcal{I\neq A}$ and $\mathcal{I}$ is the kernel of an
irreducible representation of $\mathcal{A}$ on some Hilbert space. The
primitive ideal space, $prim\left(  \mathcal{A}\right)  ,$ is the set of all
primitive ideals in $\mathcal{A}.$ For more information about $prim\left(
\mathcal{A}\right)  ,$ we refer the readers to \cite{BK1}.

\begin{lemma}
\label{strong 2}Let $\mathcal{A}$ be a separable C*-algebras. If $\mathcal{A}$
is strongly quasidiagonal, then $\mathcal{A\otimes M}_{n}\left(
\mathbb{C}\right)  $ is strongly quasidiagonal again.
\end{lemma}

\begin{proof}
By IV.3.4.25 in \cite{BK1}, we know that every primitive ideal in
$\mathcal{A\otimes M}_{n}\left(  \mathbb{C}\right)  $ is in the form
$\mathcal{I\otimes M}_{n}\left(  \mathbb{C}\right)  $ where $\mathcal{I}$ is a
primitive ideal of $\mathcal{A}.$ Hence for an arbitrary irreducible
representation $\pi$ of $\mathcal{A\otimes M}_{n}\left(  \mathbb{C}\right)  ,$
we can find a representation $\widetilde{\pi}:\mathcal{A}\rightarrow
\mathcal{B}\left(  \mathcal{H}\right)  $ and a primitive ideal $\mathcal{I}$
such that $\ker\left(  \widetilde{\pi}\right)  =\mathcal{I}$ and $\ker
(\pi)=\mathcal{I\otimes M}_{n}\left(  \mathbb{C}\right)  .$ Since
$\mathcal{A}$ is strongly quasidiagonal, we have $\widetilde{\pi}\left(
\mathcal{A}\right)  $ is a quasidiagonal set of operators in $\mathcal{B}%
\left(  \mathcal{H}\right)  $. Now we have
\[
\ker\left(  \widetilde{\pi}\otimes id\right)  =\ker\left(  \pi\right)
\]
where $id$ is the identity representation of $\mathcal{M}_{n}\left(
\mathbb{C}\right)  .$ Note $\widetilde{\pi}\otimes id$ and $\pi$ are both
irreducible, then by Corollary 4 in \cite{HD} we know that $\pi$ is
quasidiagonal precisely when $\widetilde{\pi}\otimes id$ is. By Lemma 14 in
\cite{HD}, $\widetilde{\pi}\otimes id$ is quasidiagonal since $\mathcal{A}$ is
strongly quasidiagonal. It follows that $\pi\mathcal{\ }$is quasidiagonal.
Therefore $\mathcal{A\otimes M}_{n}\left(  \mathbb{C}\right)  $ is strongly
quasidiagonal by Proposition 2.8 in \cite{BK2}.
\end{proof}

For showing our next result, we need the next concept.

\begin{definition}
(Definition 3, \cite{SL}) Let $\mathcal{S}$ be a class of C*-algebras. A
C*-algebra $\mathcal{A}$ is called a local $\mathcal{S}$-algebra if for every
finite subset $\mathcal{F\subseteq A}$ and $\forall\varepsilon>0$ there exist
a C*-algebra$\mathcal{\ B}$ in $\mathcal{S}$ and a *-homomorphism
$\varphi:\mathcal{B\rightarrow A}$ such that $dist\left(  x,\varphi\left(
\mathcal{B}\right)  \right)  <\varepsilon$ for all $x\in\mathcal{A}.$ When
$\mathcal{A}$ is unital and the C*-algebras in $\mathcal{S}$ and the $\ast
$-homomorphism are unital, the C*-algebra $\mathcal{A}$ is called unital local
$\mathcal{S}$-algebra.
\end{definition}

\begin{lemma}
\label{strong 3}If $\mathcal{S}$ is a class of strongly quasidiagonal
C*-algebras, then each local $\mathcal{S}$-algebra is strongly quasidiagonal again.
\end{lemma}

\begin{proof}
Let $\mathcal{A}$ be a local $\mathcal{S}$-algebra. Then for every finite set
$\mathcal{F}\subseteq\mathcal{A}$ and $\varepsilon>0,$ there are a strongly
quasidiagonal C*-algebra $\mathcal{B}$ and a *-homomorphism $\varphi
:\mathcal{B}\rightarrow\mathcal{A}$ such that $d\left(  a,\varphi\left(
\mathcal{B}\right)  \right)  <\varepsilon$ for every $a\in\mathcal{F}.$ Hence
we can find a finite subset $\mathcal{D}\subseteq\varphi\left(  \mathcal{B}%
\right)  $ such that for $a\in\mathcal{F}$ there exist $d\in\mathcal{D}$ such
that $\left\Vert a-d\right\Vert <\varepsilon$. Assume $\pi:\mathcal{A}%
\rightarrow$ $\mathcal{B}\left(  \mathcal{H}\right)  $ is a representation of
$\mathcal{A},$ then $\pi\circ\varphi:\mathcal{B}\rightarrow\mathcal{B}\left(
\mathcal{H}\right)  $ is a representation of $\mathcal{B}.$ Therefore
$\pi\left(  \varphi\left(  \mathcal{B}\right)  \right)  $ is a quasidiagonal
set of operators on $\mathcal{H}$ since $\mathcal{B}$ is strongly
quasidiagonal. It implies that there are a finite-rank projection
$p\in\mathcal{B}\left(  \mathcal{H}\right)  $ and a finite subset
$\mathcal{X\subset H}$ such that $\left\Vert p\pi\left(  d\right)  -\pi\left(
d\right)  p\right\Vert <\varepsilon$ and $\left\Vert p\left(  x\right)
-x\right\Vert <\varepsilon$ for every $d\in\mathcal{D}\subseteq\varphi\left(
\mathcal{B}\right)  $ and $x\in\mathcal{X}.$ Now it is easily to calculate
that for $a\in\mathcal{F}$
\[
\left\Vert p\pi\left(  a\right)  -\pi\left(  a\right)  p\right\Vert
<\left\Vert p\pi\left(  a\right)  -p\pi\left(  d\right)  \right\Vert
+\left\Vert p\pi\left(  d\right)  -\pi\left(  d\right)  p\right\Vert
+\left\Vert \pi\left(  d\right)  p-\pi\left(  a\right)  p\right\Vert
<3\varepsilon
\]
These imply that $\pi\left(  \mathcal{A}\right)  $ is a quasidiagonal set of
operators. Since $\pi$ is arbitrary, we know $\mathcal{A}$ is a strongly
quasidiagonal C*-algebras .
\end{proof}

Now we need to introduce the Rokhlin action on C*-algebras.

\begin{definition}
\label{Rokhlin 2}(Definition 2, \cite{SL}) Let $\mathcal{A}$ be a C*-algebra
and let $\alpha:G\rightarrow Aut\left(  \mathcal{A}\right)  $ be an action of
a finite group $G$ on $\mathcal{A}.$ We say that $\alpha$ has the Rokhlin
property if for any $\varepsilon>0$ and any finite subset $\mathcal{F\subseteq
A}$ there exist mutually orthogonal positive contractions $\left(
r_{g}\right)  _{g\in G}\subseteq\mathcal{A}$ such that

\begin{enumerate}
\item $\left\Vert \alpha_{g}\left(  r_{h}\right)  -r_{gh}\right\Vert
<\varepsilon$ for all $g,h\in G;$

\item $\left\Vert r_{g}a-ar_{g}\right\Vert <\varepsilon$ for all
$a\in\mathcal{F}$ and $g\in G;$

\item $\left\Vert \left(  \sum_{g\in G}r_{g}\right)  a-a\right\Vert
<\varepsilon$ for all $a\in\mathcal{F}.$
\end{enumerate}
\end{definition}

For showing our main result in this subsection, we need the following result.

\begin{theorem}
\label{Rokhlin 1}(Theorem 2, \cite{SL}) Let $\mathcal{A}$ be a C*-algebra and
$G$ be a finite group. Let $\alpha:G\rightarrow Aut\left(  \mathcal{A}\right)
$ be an action with the Roklin property. Then for any finite subset
$\mathcal{F\subseteq A\rtimes}_{\alpha}G$ and any $\varepsilon>0$ there exist
a positive element $a\in\mathcal{A}$ and a $\ast$-homomorphism $\varphi
:\mathcal{M}_{\left\vert G\right\vert }\left(  \overline{a\mathcal{A}%
a}\right)  \rightarrow\mathcal{A\rtimes}_{\alpha}G$ such that $dist\left(
x,\operatorname{Im}\left(  \varphi\right)  \right)  <\varepsilon$ for all
$x\in\mathcal{F}.$
\end{theorem}

The C*-algebra $\mathcal{A}$ in Theorem \ref{Rokhlin 1} may not be unital.
Next result is our main result in this subsection

\begin{theorem}
\label{strong 4}Let $\mathcal{A}$ be a strongly quasidiagonal C*-algebra and
$\alpha$ a Rokhlin action on $\mathcal{A}$ by a finite group $G$. Then the
crossed product $\mathcal{A\rtimes}_{\alpha}G$ is a strongly quasidiagonal
C*-algebra again.
\end{theorem}

\begin{proof}
By Lemma \ref{strong 1} and Lemma \ref{strong 2}, $\mathcal{M}_{\left\vert
G\right\vert }\left(  \overline{a\mathcal{A}a}\right)  \cong\overline
{a\mathcal{A}a}\otimes\mathcal{M}_{\left\vert G\right\vert }$ is strongly
quasidiagonal and then $\mathcal{A\rtimes}_{\alpha}G$ is locally $\mathcal{S}%
$-algebra where $\mathcal{S}$ is a class of strongly quasidiagonal C*-algebras
by Theorem \ref{Rokhlin 1}. Hence by Lemma \ref{strong 3}, $\mathcal{A\rtimes
}_{\alpha}G$ is strongly quasidiagonal.
\end{proof}

\subsection{Just-infinite C*-algebras}

In this subsection, we will see the relationship between just-infinite
C*-algebra and inner quasidiagonal C*-algebras. A C*-algebra is called
just-infinite if it is infinite-dimensional and all its proper quotient are finite-dimensional.

For each $n\in\{0,1,2,\cdots,\infty\},$ The $T_{0}$-space $Y_{n}$ is defined
to be the disjoint union $Y_{n}=\left\{  0\right\}  \cup Y_{n}^{\prime},$
where $Y_{n}^{\prime}=\left\{  1,2,\cdots,n\right\}  $ is a set with $n$
elements if $n$ is finite, and $Y_{n}^{\prime}=\mathbb{N}$ has countably
infinitely many element if $n=\infty.$ Equip $Y_{n}$ with the topology for
which the closed subsets of $Y_{n}$ are precisely the following set:
$\varnothing,$ $Y_{n}$ and all finite subsets of $Y_{n}^{\prime}.$ So
$prim\left(  \mathcal{A}\right)  $ and a just-infinite C*-algebra
$\mathcal{A}$ are characterized by using $Y_{n}$ in the following theorem.

\begin{theorem}
\label{just-infinite 1.2}(Theorem 3.10 in \cite{GMR}) Let $\mathcal{A}$ be
separable C*-algebra. Then $\mathcal{A}$ is just-infinite if and only if
$prim(\mathcal{A)}$ is homeomorphic to $Y_{n},$ for some $n\in\{0,1,2,\cdots
,\infty\},$ and each non-faithful irreducible representation of $\mathcal{A}$
is finite dimensional (If $n=0,$ we must also require that $\mathcal{A}$ is
infinite dimensional; this is automatic when $n\geq1.$) Moreover:

$(\alpha)$ $prim(\mathcal{A)}=Y_{0}$ if and only if $\mathcal{A}$ is simple.
Every infinite dimensional simple C*-algebra is just-infinite.

$(\beta)$ $prim(\mathcal{A)}=Y_{n},$ for some integer $n\geq1,$ and
$\mathcal{A}$ is just-infinite, if and only if $\mathcal{A}$ contains a simple
non-zero essential infinite dimensional ideal $\mathcal{I}_{0}$ such that
$\mathcal{A}/\mathcal{I}_{0}$ is finite dimensional. In this case, $n$ is
equal to the number of simple summand of $\mathcal{A}/\mathcal{I}_{0}.$

$(\gamma)$ The following conditions are equivalent:

$(i)$ $\mathcal{A}$ is just-infinite and $prim$($\mathcal{A)}=Y_{\infty},$

$(ii)\mathcal{A}$ is just-infinite and RFD,

$(iii)$ $prim(\mathcal{A})$ is an infinite set, all of its infinite subsets
are dense, and $\mathcal{A}/\mathcal{I}$ is finite dimensional, for each
non-zero $I\in prim(\mathcal{A}).$

$(iv)$ $prim(\mathcal{A})$ is an infinite set, the direct sum representation
$\oplus_{i}\pi_{i}$ is faithful for each infinite family $\{\pi_{i})_{i}$ of
pairwise inequivalent irreducible representations of $\mathcal{A}$, and each
non-faithful irreducible representation of $\mathcal{A}$ is finite dimensional.
\end{theorem}

\begin{lemma}
\label{just-infinite 1}(Lemma 3.2 in \cite{GMR}) Every just-infinite
C*-algebra is prime.
\end{lemma}

It implies that every non-zero ideal $\mathcal{J}$ in just-infinite
C*-algebras $\mathcal{A}$ is essential, i.e., $\mathcal{I\cap J}$
$\neq\left\{  0\right\}  $ for every non-zero closed ideal $\mathcal{I}$ in
$\mathcal{A}.$

By the similar arguments in the previous subsection, we can quickly get the
following two results about just-infinite C*-algebras.

\begin{proposition}
\label{just1}Let $\mathcal{A}$ be a separable C*-algebra. If $\mathcal{A}$ is
just-infinite, then $\mathcal{A\otimes M}_{n}\left(  \mathbb{C}\right)  $ is
just-infinite again.
\end{proposition}

\begin{proof}
Assume $\pi$ is a faithful irreducible representation of $\mathcal{A}.$ Then
$\pi\otimes id$ is a faithful irreducible representation of $\mathcal{A\otimes
M}_{n}\left(  \mathbb{C}\right)  $ where $id$ is the identity representation
on $\mathcal{M}_{n}\left(  \mathbb{C}\right)  .$ It implies that
$\mathcal{A\otimes M}_{n}\left(  \mathbb{C}\right)  $ is prime. By Blackadar
IV.3.4.25, every primitive ideal $\mathcal{J}$ in $\mathcal{A\otimes M}%
_{n}\left(  \mathbb{C}\right)  $ is in the form $\mathcal{I\otimes M}%
_{n}\left(  \mathbb{C}\right)  $ where $\mathcal{I}$ is a primitive ideal of
$\mathcal{A}$ and $\mathcal{A}/\mathcal{I}$ is finite-dimensional as
$\mathcal{I}$ is nontrivial. It implies that
\[
(\mathcal{A\otimes M}_{n}\left(  \mathbb{C}\right)  )/\mathcal{J}\cong\left(
\mathcal{A}/\mathcal{I}\right)  \otimes\mathcal{M}_{n}\left(  \mathbb{C}%
\right)
\]
is finite-dimensional. So by Theorem \ref{just-infinite 1.2},
\[
Prim\left(  \mathcal{A\otimes M}_{n}\left(  \mathbb{C}\right)  \right)  \cong
Prim\left(  \mathcal{A}\right)  \cong Y_{n}%
\]
for some $n\in\left\{  0,1,2,\cdots,\infty\right\}  $ and then
$\mathcal{A\otimes M}_{n}\left(  \mathbb{C}\right)  $ is just-infinite dimensional.
\end{proof}

\begin{proposition}
\label{just2}Let $\mathcal{A}$ be a just-infinite C*-algebra and
$\mathcal{B}\subseteq\mathcal{A}$ a hereditary C*-subalgebra. Then
$\mathcal{B}$ is just-infinite again.
\end{proposition}

\begin{proof}
By Blackadar II.6.1.9, if $\mathcal{A}$ is prime, then $\mathcal{B}$ is prime
too. Since
\[
\mathcal{B}/\left(  \mathcal{B\cap I}\right)  \mathcal{\subseteq}%
\mathcal{A}/\mathcal{I},
\]
we know that $\mathcal{B}/\left(  \mathcal{B\cap I}\right)  $ is
finite-dimensional. Then by Blackadar II.5.3.5, we know that $\mathcal{B}$ is just-infinite.
\end{proof}

Now we are going to show that a just-finite C*-algebra is quasidiagonal if and
only if it is inner quasidiagonal.

\begin{theorem}
\label{just3}Let $\mathcal{A}$ be a separable C*-algebra. If $\mathcal{A}$ is
just-infinite, then $\mathcal{A}$ is quasidiagonal if and only if it is inner quasidiagonal.
\end{theorem}

\begin{proof}
Since every inner quasidiagonal C*-algebra is quasidiagonal. So one direction
of the proof is clear. Now we only need to prove that a just-infinite
C*-algebra $\mathcal{A}$ is inner quasidiagonal as $\mathcal{A}$ is quasidiagonal.

Since $\mathcal{A}$ is separable, $\mathcal{A}$ is just-infinite if and only
if it is in the three cases listed in Theorem \ref{just-infinite 1.2}. In the
case $(\alpha)$, we already know $\mathcal{A}$ is inner quasidiagonal by Lemma
\ref{inner 1.2}. In the case $(\gamma)$, it is obvious that $\mathcal{A}$ is
inner quasidiagonal by Remark \ref{inner 1}. Hence we only need to verify that
$\mathcal{A}$ is inner quasidiagonal in the case $(\beta)$.

By Lemma \ref{just-infinite 1}, we know $\mathcal{A}$ is prime and then every
ideal in $\mathcal{A}$ is essential. Now we consider two cases. First case, we
suppose $\mathcal{A}$ is antiliminal. Then $\mathcal{A}$ is inner
quasidiagonal by Lemma \ref{inner 1.2}. In the second case, we suppose
$\mathcal{A}$ is not antiliminal. Since $\mathcal{A}$ has a faithful
irreducible representation $\pi$ and there is an element $x\in\mathcal{A}$
such that $rank\left(  \pi\left(  x\right)  \right)  \leq1$ by IV.1.1.7 in
\cite{BK1}, then $\mathcal{A}$ has an essential ideal isomorphic to
$\mathbb{K}$. So we may assume $\mathbb{K}\subseteq\mathcal{A}.$ Note
$\mathcal{I}_{0}$ in Theorem \ref{just-infinite 1.2} is essential, so we know
that $\mathbb{K}\cap\mathcal{I}_{0}\neq\{0\}$. by the fact that $\mathcal{I}$
is simple, we have%
\[
\mathbb{K}\cap\mathcal{I}_{0}=\mathbb{K=}\mathcal{I}_{0}.
\]
It follows that $\mathcal{A}$ is an extension of two AF-algebras, it is itself
an AF-algebra. By Remark \ref{inner 1}, $\mathcal{A}$ is inner quasidiagonal.
\end{proof}

\smallskip A separable C*- algebra $\mathcal{A}$ is called MF if it can be
embedded into%
\[%
{\textstyle\prod\limits_{k}}
\mathcal{M}_{n_{k}}\left(  \mathbb{C}\right)  /\sum_{k}\mathcal{M}_{n_{k}%
}\left(  \mathbb{C}\right)
\]
for a sequence of positive integers $\left\{  n_{k}\right\}  _{k=1}%
^{\mathcal{1}}$. Since every nuclear MF algebra is quasidiagonal, we can get
the next result.

\begin{corollary}
Let $\mathcal{A}$ be a unital just-infinite MF-algebra$.$ If $\mathcal{A}$ is
nuclear, then $\mathcal{A}$ is inner quasidiagonal.
\end{corollary}

Combining Proposition \ref{just2}, Proposition \ref{inner here} and Theorem
\ref{just3}, we can quickly get the following corollary.

\begin{corollary}
\label{just4}Let $\mathcal{A}$ be a separable C*-algebra. If $\mathcal{A}$ is
a just-infinite quasidiagonal C*-algebra and $\mathcal{B}\subseteq\mathcal{A}$
is a hereditary C*-subalgebra, then $\mathcal{B}$ is just-infinite
quasidiagonal again.
\end{corollary}

\begin{remark}
A C*-algebra is called FDI if it has no infinite-dimensional irreducible
representation (\cite{CS}). By Theorem \ref{just3} every just-infinite
quasidiagonal C*-algebras is inner quasidiagonal in the separable case.
Meanwhile, it is easy to see that every FDI C*-algebra is inner quasidiagonal,
and we also know that no FDI C*-algebra is just-infinite by the fact that
just-infinite C*-algebras are prime. Therefore in the separable case the set
of just-infinite quasidiagonal C*-algebras has no intersection with the set of
FDI C*-algebra even though they are both inner quasidiagonal.
\end{remark}

The following result is an analogous conclusion of Theorem \ref{strong 4}.

\begin{theorem}
\label{just5}Let $\mathcal{A}$ be a just-infinite C*-algebra and
$\alpha:G\rightarrow Aut\left(  \mathcal{A}\right)  $ be an action of a finite
group $G$ on $\mathcal{A}$ with the Rokhlin property. Then $\mathcal{A}%
\rtimes_{\alpha}G$ is just-infinite.
\end{theorem}

\begin{proof}
By Proposition 3 in \cite{SL}, we know every nontrivial ideal $\mathcal{J}$ in
$\mathcal{A}\rtimes_{\alpha}G$ has the form $\mathcal{I}\rtimes_{\alpha\left(
\cdot\right)  |_{\mathcal{I}}}G$ for some $G$-invariant ideal $\mathcal{I}$ of
$\mathcal{A}.$ Therefore by Remark 7.14 in \cite{W}, we have $\left(
\mathcal{A}\rtimes_{\alpha}G\right)  /\mathcal{J}\cong\left(  \mathcal{A}%
/\mathcal{I}\right)  \rtimes_{\alpha\left(  \cdot\right)  |_{\mathcal{I}}}G$
is finite-dimensional. It implies that $\mathcal{A}\rtimes_{\alpha}G$ is just-infinite.
\end{proof}

\begin{corollary}
Let $\mathcal{A}$ be a just-infinite quasidiagonal C*-algebra and
$\alpha:G\rightarrow Aut\left(  \mathcal{A}\right)  $ be an action of a finite
group $G$ on $\mathcal{A}$ with the Rokhlin property. Then $\mathcal{A}%
\rtimes_{\alpha}G$ is just-infinite quasidiagonal.
\end{corollary}

\begin{proof}
Since $\mathcal{A}$ is quasidiagonal, we know $\mathcal{A}\rtimes_{\alpha
}G\subseteq\mathcal{A\otimes M}_{|G|}\left(  \mathbb{C}\right)  $ is
quasidiagonal by V.4.2.5 in \cite{BK1}. Hence from Theorem \ref{just5},
$\mathcal{A}\rtimes_{\alpha}G$ is just-infinite quasidiagonal.
\end{proof}

\section{Topological free entropy dimension in just-infinite C*-algebras}

In this section, we will first recall the notion of topological free entropy
dimension $\delta_{top}\left(  x_{1},\cdots,x_{n}\right)  $ for $n$-tuple
$\overrightarrow{x}=\left(  x_{1},\cdots,x_{n}\right)  $ in a unital
C*-algebra. After that, we are going to analyze the topological free entropy
dimension in just-infinite C*-algebras.

\subsection{\bigskip Definition of $\delta_{top}$ for $n$-tuple}

The concept of MF algebras was first introduced by Blackadar and Kirchberg in
\cite{BK}. The topological free entropy dimension is well-defined for the
$n$-tuple $\left(  x_{1},\ldots,x_{n}\right)  $ if the unital C*-algebra
generated by $\left\{  x_{1},\ldots,x_{n}\right\}  $ is an MF C*-algebra.

\begin{definition}
\textbf{The topological free entropy dimension }of $x_{1},\ldots,x_{n}$ is
defined by%
\begin{align*}
&  \delta_{\text{top}}(x_{1},\ldots,x_{n})\\
&  =\underset{\omega\rightarrow0^{+}}{\lim\sup}\underset{\varepsilon
>0,r\in\mathbb{N}}{\inf}\underset{k\rightarrow\mathcal{1}}{\lim\sup}\frac
{\log(\nu_{\mathcal{1}}(\Gamma^{\text{(top)}}(x_{1},\ldots,x_{n}%
;k,\varepsilon,Q_{1},\ldots,Q_{r}),\omega))}{-k^{2}\log\omega}.
\end{align*}

\end{definition}

For the real definition of topological free entropy dimension, we refer the
readers to \cite{DV}, \cite{HLS} or \cite{HLSW}. One of main questions about
topological free entropy dimension is whether it is independent of the
generators of unital MF C*-algebras. In \cite{HLSW}, we can find the following results.

\begin{theorem}
\label{don}(Theorem 4.5,\cite{HLSW}) If $\mathcal{A=}C^{\ast}\left(
x_{1},\cdots,x_{n}\right)  $ is MF-nuclear, then
\[
\delta_{top}\left(  x_{1},\cdots,x_{n}\right)  \leq1.
\]

\end{theorem}

\begin{theorem}
\label{top 3}(Theorem 3.6 in \cite{HLSW}) Suppose $\mathcal{A}=C^{\ast}%
(x_{1},\cdots,x_{n})$ is an MF-algebra and $\mathcal{A}/\mathcal{J}%
_{MF}(\mathcal{A)}$ has dimensional $d<\infty.$ Then $\delta_{top}%
(x_{1},\cdots,x_{n})=1-\frac{1}{d}$
\end{theorem}

The previous theorem tell us that the topological free entropy dimension of
$\mathcal{A}$ is independent of the generators of $\mathcal{A}$ if
$\mathcal{A}/\mathcal{J}_{MF}(\mathcal{A)}$ has dimensional $d<\infty.$

\begin{lemma}
\label{top 4}(Theorem 5.3,\cite{HLSW}). Suppose $\mathcal{A}=C^{\ast}%
(x_{1},\cdots,x_{n})$ is a unital MF-algebra and either

(i) $\mathcal{A}$ has no finite-dimensional representations, or

(ii) $\mathcal{A}$ has infinitely many non-unitarily-equivalent
finite-dimensional irreducible representation.

Then $\delta_{top}(x_{1},\cdots,x_{n})\geq1.$
\end{lemma}

By Lemma \ref{top 4}, we can quickly get the following corollary.

\begin{corollary}
\label{top 5}Let $\mathcal{A}=C^{\ast}(x_{1},\cdots,x_{n})$ be a simple
infinite-dimensional MF algebra, then $\delta_{top}(x_{1},\cdots,x_{n})\geq1.$
\end{corollary}

\begin{remark}
\label{top 9} In the von Neumann version of free entropy dimension, Voiculescu
in \cite{V3} prove that if $x=x^{\ast}$ is an element of a von Neumann algebra
with faithful trace $\tau$, then%
\[
\delta_{0}\left(  x\right)  =1-\sum_{t\text{ is an eigenvalue of }x}%
\tau\left(  P_{t}\right)  ^{2}%
\]
where $P_{t}$ is the orthogonal projection onto $\ker\left(  x-t\right)  .$ If
$x$ has no eigenvalues, then $\delta_{0}\left(  x\right)  =1.$
\end{remark}

Suppose $\tau$ is a tracial state of $\mathcal{A}$ and $\pi_{\tau}$ is the GNS
construction induced by $\tau$ with cyclic vector $e.$ Then $\widehat{\tau
}:\pi\left(  \mathcal{A}\right)  ^{\prime\prime}\rightarrow\mathbb{C}$ is the
faithful tracial state defined by $\widehat{\tau}\left(  T\right)
=\left\langle Te,e\right\rangle .$ So we can find next lemma in \cite{HLSW}.

\begin{lemma}
\label{top 8}Suppose $\mathcal{A}=C^{\ast}\left(  x_{1},\cdots,x_{n}\right)  $
is a unital MF-aglebra and $\tau\in\mathcal{T}_{MF}\left(  \mathcal{A}\right)
.$ Suppose $b=b^{\ast}\in\pi_{\tau}\left(  \mathcal{A}\right)  ^{\prime\prime
}$. Then
\[
\delta_{top}\left(  x_{1},\cdots,x_{n}\right)  >\delta_{0}\left(
b,\widehat{\tau}\right)  .
\]

\end{lemma}

\begin{definition}
(\cite{HLSW}) \label{33}Suppose $\mathcal{A=}C^{\ast}\left(  x_{1}%
,\cdots,x_{n}\right)  $ is an MF C*-algebra. A tracial state $\tau$ on
$\mathcal{A}$ is an MF-trace if there is sequence $\left\{  m_{k}\right\}  $
of positive integers and sequences $\left\{  A_{1k}\right\}  ,\cdots,\left\{
A_{nk}\right\}  $ with $A_{1k},\cdots,A_{nk}\in\mathcal{M}_{m_{k}}\left(
\mathbb{C}\right)  $ such that, for every $\ast$-polynomial $Q,$

\begin{enumerate}
\item $\lim_{k\rightarrow\mathcal{1}}\left\Vert Q\left(  A_{1k},\cdots
,A_{nk}\right)  \right\Vert =\left\Vert Q\left(  x_{1},\cdots,x_{n}\right)
\right\Vert ,$ and

\item $\lim_{k\rightarrow\mathcal{1}}\tau_{m_{k}}\left(  Q\left(
A_{1k},\cdots,A_{nk}\right)  \right)  =\tau\left(  Q\left(  x_{1},\cdots
,x_{n}\right)  \right)  .$
\end{enumerate}
\end{definition}

We let $\mathcal{TS}\left(  \mathcal{A}\right)  $ denote the set of all
tracial states on $\mathcal{A}$ and $\mathcal{T}_{MF}\left(  \mathcal{A}%
\right)  $ denote the set of all MF-traces on $\mathcal{A}.$ The MF-ideal of
$\mathcal{A}$ is defined as
\[
\mathcal{J}_{MF}\left(  \mathcal{A}\right)  =\left\{  a\in\mathcal{A}%
:\forall\tau\in\mathcal{T}_{MF}\left(  \mathcal{A}\right)  ,\tau\left(
a^{\ast}a\right)  =0\right\}  .
\]

It has been shown that $\mathcal{J}_{MF}\left(  \mathcal{A}\right)  $ is a
nonempty weak*-compact convex set in \cite{HLSW}. Furthermore, we let
$\mathcal{S}$ denote the class of MF C*-algebras $\mathcal{A}$ for which every
trace is an MF-trace, i.e., $\mathcal{TS}\left(  \mathcal{A}\right)
=\mathcal{T}_{MF}\left(  \mathcal{A}\right)  $ and $\mathcal{W}$ denote the
class of all MF algebras $\mathcal{A}$ such that $\mathcal{J}_{MF}\left(
\mathcal{A}\right)  =\left\{  0\right\}  .$ For more details about
$\mathcal{S}$ and $\mathcal{W}$, we refer the readers to \cite{HLSW}.

\subsection{$\delta_{top}$ in just-infinite C*-algebras}

The following two results are easy consequences of the properties of
just-infinite C*-algebras and $\delta_{top}.$

\begin{proposition}
\label{top 6}Let $\mathcal{A}=C^{\ast}(x_{1},\cdots,x_{n})$ be a unital
just-infinite MF C*-algebra$.$ If $\mathcal{A}\notin\mathcal{W}$, then
$\delta_{top}(x_{1},\cdots,x_{n})=1-\frac{1}{\dim\left(  \mathcal{A}%
\text{/}\mathcal{J}_{MF}\left(  \mathcal{A}\right)  \right)  }.$
\end{proposition}

\begin{proof}
If $\mathcal{A\notin W}$, then $\mathcal{J}_{MF}\left(  \mathcal{A}\right)
\neq0.$ Since $\mathcal{A}$ is just-finite, we have $\mathcal{A}%
/\mathcal{J}_{MF}\left(  \mathcal{A}\right)  $ is finite-dimensional. Then
$\delta_{top}(x_{1},\cdots,x_{n})=1-\frac{1}{\dim\left(  \mathcal{A}%
\text{/}\mathcal{J}_{MF}\left(  \mathcal{A}\right)  \right)  }$by Theorem
\ref{top 3}.
\end{proof}

\begin{proposition}
\label{top 7}Let $\mathcal{A}=C^{\ast}(x_{1},\cdots,x_{n})$ be a unital
just-infinite MF C*-algebra$.$ If $\mathcal{A}$ is of type $(\alpha)$ or type
$(\gamma)$ in Theorem \ref{just-infinite 1.2}, then $\delta_{top}(x_{1}%
,\cdots,x_{n})\geq1.$
\end{proposition}

\begin{proof}
Suppose $\mathcal{A}$ is of type $(\alpha),$ then $\mathcal{A}$ is simple and
infinite-dimensional. By Corollary \ref{top 5}, $\delta_{top}(x_{1}%
,\cdots,x_{n})\geq1.$ Suppose $\mathcal{A}$ is of type $(\gamma),$ then
$\mathcal{A}$ has infinite many non-unitarily-equivalent finite dimensional
irreducible representations, so $\delta_{top}(x_{1},\cdots,x_{n})\geq1$ by
Lemma \ref{top 4}.
\end{proof}

\begin{remark}
\label{top 1}Note if $\mathcal{A}$ is of type $(\alpha)$ or type $(\gamma),$
then $\mathcal{A}\in\mathcal{W}.$
\end{remark}

\begin{theorem}
Let $\mathcal{A}=C^{\ast}(x_{1},\cdots,x_{n})$ be a unital MF C*-algebra$.$ If
$\mathcal{A}$ is just-infinite and $\mathcal{A}\in\mathcal{W\cap S}$ , then
$\delta_{top}(x_{1},\cdots,x_{n})\geq1.$
\end{theorem}

\begin{proof}
By Remark \ref{top 1} and Proposition \ref{top 7}, we know that $\delta
_{top}(x_{1},\cdots,x_{n})\geq1$ as $\mathcal{A}$ is of type $(\alpha)$ or
$(\gamma).$ So we only need to consider the case in which $\mathcal{A}$ is
type $(\beta)$ in Theorem \ref{just-infinite 1.2}.

Assume $\mathcal{A}$ is not antiliminal. Since $\mathcal{A}$ is prime,
$\mathcal{A}$ has a faithful irreducible representation $\pi$ and there is an
element $x\in\mathcal{A}$ such that $rank\left(  \pi\left(  x\right)  \right)
\leq1$ by IV.1.1.7 in \cite{BK1}. It follows that $\mathcal{A}$ has an
essential ideal isomorphic to $\mathbb{K}$. So we may assume $\mathbb{K}%
\subseteq\mathcal{A}.$ Note $\mathcal{I}_{0}$ in Theorem
\ref{just-infinite 1.2} is essential, so we know that $\mathbb{K}%
\cap\mathcal{I}_{0}\neq\{0\}$. By the fact that $\mathcal{I}_{0}$ is simple,
we have
\[
\mathbb{K}\cap\mathcal{I}_{0}=\mathbb{K=}\mathcal{I}_{0}.
\]
Since $\mathcal{A}/\mathcal{I}_{0}$ is finite-dimensional and every tracial
state vanish on $\mathbb{K}$, we have
\[
\mathbb{K}=\mathcal{J}_{MF}\left(  \mathcal{A}\right)  \neq0.
\]
This contradict to the fact that $\mathcal{A}\in\mathcal{W}$. Then
$\mathcal{A}$ must be antiliminal$.$ Let $\tau$ be a factor tracial state on
$\mathcal{A}$. Then
\begin{equation}
\ker\pi_{\tau}\supseteq\mathcal{I}_{0}\text{ or}\ker\pi_{\tau}=0
\tag{3.1}\label{3.1}%
\end{equation}
by the fact that $\mathcal{I}_{0}$ is essential and simple where $\pi_{\tau}$
is a representation induced by $\tau$. Note $\mathcal{A\in W\cap S},$ then
every factor tracial state is MF tracial state and
\begin{equation}
\mathcal{J}_{MF}\left(  \mathcal{A}\right)  =0. \tag{3.2}\label{3.2}%
\end{equation}
In \cite{HM}, it was shown that the set of factor tracial states is the set of
extreme points of $\mathcal{TS}\left(  \mathcal{A}\right)  $. It follows that
there is at lease one factor tracial state $\tau$ such that $\pi_{\tau}$ is
faithful by (3.1) and (3.2). Then for such $\tau,$ $\pi_{\tau}\left(
\mathcal{A}\right)  $ is antiliminal too. If $\pi_{\tau}\left(  \mathcal{A}%
\right)  ^{\prime\prime}$ is type I factor with tracial state, then $\pi
_{\tau}\left(  \mathcal{A}\right)  =\pi_{\tau}\left(  \mathcal{A}\right)
^{\prime\prime}\cong\mathcal{M}_{n}\left(  \mathbb{C}\right)  \otimes I_{k}$
for some integer $n$ and $k.$ This contradict to the fact that $\pi_{\tau
}\left(  \mathcal{A}\right)  $ is antiliminal. Hence $\pi_{\tau}\left(
\mathcal{A}\right)  ^{\prime\prime}$ is type II$_{1}$ factor. It implies that
there is an element $a\in\pi_{\tau}\left(  \mathcal{A}\right)  ^{\prime\prime
}$ which has no eigenvalues, so $\delta_{0}\left(  a\right)  =1$ by Remark
\ref{top 9}$.$ Now by Lemma \ref{top 8},
\[
\delta_{top}\left(  x_{1},\cdots,x_{n}\right)  \geq1.
\]

\end{proof}

Combining preceding results and Theorem \ref{don}, we have the following two corollaries.

\begin{corollary}
Let $\mathcal{A}=C^{\ast}(x_{1},\cdots,x_{n})$ be a unital C*-algebra. If
$\mathcal{A}$ is just-infinite MF-nuclear and $\mathcal{A}\in\mathcal{W\cap
S}$, then $\delta_{top}(x_{1},\cdots,x_{n})=1.$
\end{corollary}

\begin{corollary}
Let $\mathcal{A}=C^{\ast}(x_{1},\cdots,x_{n})$ be a unital C*-algebra. If
$\mathcal{A}$ is just-infinite MF-nuclear$\ $and is of type $(\alpha)$ or type
$(\gamma)$, then $\delta_{top}(x_{1},\cdots,x_{n})=1.$
\end{corollary}

\end{document}